\documentclass[12pt,reqno]{amsart}
\usepackage{graphicx,subfigure}
\usepackage{hyperref}
\newcommand{\vertiii}[1]{{\left\vert\kern-0.25ex\left\vert\kern-0.25ex\left\vert #1
    \right\vert\kern-0.25ex\right\vert\kern-0.25ex\right\vert}}
\hypersetup{colorlinks=true, citecolor=blue, linkcolor=red}

\numberwithin{equation}{section} \numberwithin{figure}{section}
\numberwithin{table}{section} 
{ \theoremstyle{plain}
\newtheorem{theorem}{Theorem}[section]

\newtheorem{lemma}[theorem]{Lemma}

}

\begin{document}

\title[On the growth of the energy  of entire solutions]{On the growth of the energy  of entire solutions to the vector  Allen-Cahn equation}



\author{Christos Sourdis} \address{Department of Mathematics and Applied Mathematics, University of
Crete.}
              \email{csourdis@tem.uoc.gr}           




\maketitle

\begin{abstract}
We prove that the energy over
 balls of entire, nonconstant, bounded solutions to the vector Allen-Cahn equation grows faster than $(\ln R)^k R^{n-2}$, for any $k>0$, as the volume $R^n$ of the ball tends to infinity. This improves the  growth rate of order $R^{n-2}$ that
 follows from the general weak monotonicity formula. Moreover, our estimate can be considered as an approximation to the corresponding rate of order $R^{n-1}$ that is known to hold in the scalar case.
\end{abstract}
\section{Introduction}


Consider the semilinear elliptic system
\begin{equation}\label{eqEq}
\Delta u= W_u(u)\ \ \textrm{in}\ \ \mathbb{R}^n,\ \ n\geq 2,
\end{equation}
where $W:\mathbb{R}^m\to \mathbb{R}$, $m\geq 2$, is sufficiently
smooth and \emph{nonnegative}. This system has variational structure, as
solutions (in a smooth, bounded  domain $\Omega\subset \mathbb{R}^n$)
are critical points of the energy
\[
E(v;\Omega)=\int_{\Omega}^{}\left\{\frac{1}{2}|\nabla v|^2+ W(v)
\right\}dy
\]
(subject to their own boundary conditions).

In the scalar case, namely $m=1$, Modica \cite{modica} used the
$P$-function technique \cite{sperb} and intrinsically scalar arguments to show that every entire, bounded solution to
(\ref{eqEq}) satisfies the pointwise gradient bound
\begin{equation}\label{eqmodica}
\frac{1}{2}|\nabla u|^2\leq W(u) \ \ \textrm{in}\ \ \mathbb{R}^n,
\end{equation}
(see also \cite{cafamodica} and \cite{farinaFlat}). Using this,
together with Pohozaev identities, it was shown in \cite{modicaProc}
that  such solutions satisfy the following
strong monotonicity formula:
\begin{equation}\label{eqmonotonicity}
\frac{d}{dr}\left(\frac{1}{r^{n-1}}\int_{B(x,r)}^{}\left\{\frac{1}{2}|\nabla
u|^2+ W\left(u\right) \right\}dy\right)\geq 0,\ \ r>0,\ x\in
\mathbb{R}^n,
\end{equation}
where $B(x,r)$ stands for the $n$-dimensional ball of radius $r$
that is centered at  $x$. In particular, it follows that  each entire,
bounded and nonconstant solution to the scalar problem satisfies
\begin{equation}\label{eqenergyLowerModi}
\int_{B(x,r)}^{}\left\{\frac{1}{2}|\nabla v|^2+ W(v)
\right\}dy\geq c r^{n-1},\ \ r>0,\ \
\textrm{for some}\ c>0.
\end{equation}

In the vector case, that is $m\geq 2$,
the analog of the gradient bound (\ref{eqmodica}) does not hold in general.
In passing, let us mention that if a solution $u$ satisfied the analog of the gradient bound
(\ref{eqmodica}), then it would also satisfy the strong monotonicity formula (\ref{eqmonotonicity})
(see \cite{alikakosBasicFacts}, \cite{sourdisMonot}).
All  is not lost, however,
as using the fact that every solution to (\ref{eqEq}) satisfies the weak monotonicity formula
\begin{equation}\label{eqmonotoniWeak}
\frac{d}{dr}\left(\frac{1}{r^{n-2}}\int_{B(x,r)}^{}\left\{\frac{1}{2}|\nabla
u|^2+ W\left(u\right) \right\}dy\right)\geq 0,\ \ r>0,\ x\in
\mathbb{R}^n, \ \ (\textrm{see\ \cite{alikakosBasicFacts,bethuelBIG,caffareliLin,smets,sourdisMonot}}),
\end{equation}
and with some more work in the case $n=2$, it follows readily that, given $x\in \mathbb{R}^n$,  each
nonconstant solution to the system (\ref{eqEq}) satisfies:
\begin{equation}\label{eqGrande}
\int_{B(x,r)}^{}\left\{\frac{1}{2}|\nabla u|^2+ W\left(u\right)
\right\}dy\geq \left\{\begin{array}{ll}
                        c r^{n-2} & \textrm{if}\ n\geq 3,   \\
                          &     \\
                        c \ln r  & \textrm{if}\ n=2,
                      \end{array}
\right.
\end{equation}
for all $r>1$ and some $c>0$.

Let us mention that the above lower bound is sharp in the case $n=2$.
Indeed, for the
Ginzburg-Landau system
\begin{equation}\label{eqGL}
\Delta u=\left(|u|^2-1 \right)u,\ \ u:\mathbb{R}^n\to \mathbb{R}^m,
\ \ \left(\textrm{here}\ W(u)=\frac{\left(1-|u|^2\right)^2}{4}\
\textrm{vanishes on}\ \mathbb{S}^{m-1} \right),
\end{equation}
arising in superconductivity, there are entire, bounded solutions with   energy  over
$B(0,r)$ of order $\ln r$, as $r\to \infty$, if $n=m=2$ (see
\cite{bethuelBIG} and the references therein).

In this note, we will specialize to the class of  potentials that satisfy the following properties:
$W\in C^2$ and there exist $N\in \mathbb{N}$ points $a_i\in \mathbb{R}^m$ such that
\begin{equation}\label{eqpot1}
W(a_i)=0,\ \ W>0\ \ \textrm{in}\ \ \mathbb{R}^m\setminus \{a_i \}\ \ \textrm{and}\ \
W_{uu}(a_i)\nu \cdot \nu >0\ \ \forall\ \nu \in \mathbb{S}^{m-1},\ \ i=1,\cdots, N,
\end{equation}
where $\cdot$ stands for the Euclidean inner product in $\mathbb{R}^m$.
For this class of potentials, the system (\ref{eqEq}) is known as the vector Allen-Cahn equation and
models multiphase transitions (see \cite{baldo} \cite{bronReih}).

In this case, it was shown recently in \cite{AlikakosDensity}, by extending the density estimates of \cite{cafaCordoba} to this vector setting, that nonconstant bounded, \emph{minimal} solutions satisfy
\begin{equation}\label{eqAF}
\liminf_{r\to
\infty}\frac{1}{r^{n-1}}\int_{B(x,r)}^{}\left\{\frac{1}{2}|\nabla
u|^2+ W\left(u\right)\right\}dy>0,\ \ \forall\ x\in \mathbb{R}^n.
\end{equation}
For a simple proof of this result when $n=2$, under weaker assumptions on $W$, we refer to \cite{sourdis14}.
In comparison, let us note that the solution mentioned in relation to (\ref{eqGL}) is minimal.
On the other side, if $u$ is a nonconstant solution which is periodic in each variable,  as in \cite{batesCrystal} for the vector Allen-Cahn equation or \cite{mironescPeriodic} for (\ref{eqGL}), it is easy to see that
\[
\liminf_{r\to
\infty}\frac{1}{r^{n}}\int_{B(x,r)}^{}\left\{\frac{1}{2}|\nabla
u|^2+ W\left(u\right)\right\}dy>0,\ \ \forall\ x\in \mathbb{R}^n.
\]
It was shown recently in \cite{sourdis14} that the above relation also holds for nonconstant radial solutions, as the ones in \cite{guiRadial} for the scalar Allen-Cahn equation $\Delta u=u^3-u$.

It is therefore natural to ask what can be said, besides of the
lower bound (\ref{eqGrande}), about arbitrary bounded solutions to
the vector Allen-Cahn equation. Our main result is the following
improvement of the lower bound (\ref{eqGrande}) for this class of systems.

\begin{theorem}\label{thmMine}
Assume that $W\in C^2(\mathbb{R}^m;\mathbb{R})$, $m\geq 2$, satisfies (\ref{eqpot1}).
 If $u\in C^2(\mathbb{R}^n;\mathbb{R}^m)$, $n\geq 2$, is
a  nonconstant, entire and bounded solution to the
elliptic system (\ref{eqEq}),
for any $x\in \mathbb{R}^n$ and $k>0$, it holds that
 \begin{equation}\label{eqassert2}
\frac{1}{(\ln r)^k}\frac{1}{ r^{n-2}}\int_{B(x,r)}^{}\left\{\frac{1}{2}|\nabla u|^2+ W\left(u\right)
\right\}dy\to \infty,\ \ \textrm{as}\ \ r\to \infty.
\end{equation}
\end{theorem}

Our result
implies that, in contrast to the Ginzburg-Landau system
(\ref{eqGL}), the growth in the latter estimate cannot be achieved
for any nonconstant bounded solution of the vector Allen-Cahn
equation for any $n,m\geq 2$.
Moreover, it can be cosidered as an approximation to the corresponding lower bound (\ref{eqenergyLowerModi})
that holds in the scalar case and to (\ref{eqAF}) in the case of minimal solutions.

Our proof of Theorem \ref{thmMine} will be based on the monotonicity formula (\ref{eqmonotoniWeak}) and the following lemma from \cite{smets}:
\begin{lemma}\label{lemSmets}
Assume that $W\in C^1(\mathbb{R}^m;\mathbb{R})$, $m\geq 2$,  and that
$u\in C^2(\mathbb{R}^n;\mathbb{R}^m)$ satisfies (\ref{eqEq}).
Then, for every  $x\in \mathbb{R}^n$ and any positive numbers $R_0<R_1$, there exists $r(x)\in (R_0,R_1)$ such that
\begin{equation}\label{eqsmets}
r(x)\int_{\partial B\left(x,r(x) \right)}^{}
\left|\frac{\partial u}{\partial \nu} \right|^2 dS(y)
+2\int_{B\left(x,r(x) \right)}^{}W(u)dy\leq \frac{1}{\ln(R_1/R_0)}\ln\left( \frac{\tilde{E}(u,x,R_1)}{\tilde{E}(u,x,R_0)}\right)\int_{B\left(x,r(x) \right)}^{}e(u)dy,
\end{equation}
where
\[
e(u)\equiv\frac{1}{2}|\nabla u|^2+ W\left(u\right)
\]
and
\[
\tilde{E}(u,x,r)\equiv\frac{1}{r^{n-2}}\int_{B(x,r)}^{}\left\{\frac{1}{2}|\nabla u|^2+ W\left(u\right)
\right\}dy.
\]
\end{lemma}

This lemma was  proven in \cite{smets} for the special case of the Ginzburg-Landau system (\ref{eqGL}) but  the proof carries over verbatim to the general case. In particular, it is based on the  identity
\begin{equation}\label{eqmonotSmets}
\frac{d}{dr}\left(\tilde{E}(u,x,r) \right)=\frac{1}{r^{n-2}}\int_{\partial B\left(x,r \right)}^{}
\left|\frac{\partial u}{\partial \nu} \right|^2 dS(y)+\frac{2}{r^{n-1}}\int_{ B\left(x,r \right)}^{}
W(u)dy,
\end{equation}
which implies at once the weak monotonicity formula
(\ref{eqmonotoniWeak}) for nonnegative $W$ and is a direct consequence of Pohozaev identities for systems (see \cite{zhao} for the
case of general $W$).
Note also that, in contrast to \cite{smets}, we have included the boundary integral in the lefthand side of (\ref{eqsmets}) (we have also kept the factor $2$).

The rest of this article is devoted to the
proof of Theorem \ref{thmMine}.

\section{Proof of the main result}
\begin{proof}[Proof of Theorem \ref{thmMine}]
Since the problem is translation invariant, without loss of
generality, we may carry out the proof for $x=0$.

For future reference, we note that by standard elliptic regularity theory (see \cite{Gilbarg-Trudinger}), there exists a constant $C_0>0$ such that
\begin{equation}\label{eqapriori}
|u|+|\nabla u|\leq C_0\ \ \textrm{in}\ \ \mathbb{R}^n.
\end{equation}

Suppose, to the contrary, that (\ref{eqassert2}) does not hold.
Then, there would exist   constants $k,C_1>1$ and a sequence $R_j>1$ with $R_j\to \infty$ such that
\begin{equation}\label{eqoriginal}
\tilde{E}(u,0,R_j)\leq C_1 (\ln R_j)^k,\ \ j\geq 1.
\end{equation}
On the other side, thanks to (\ref{eqGrande}), we have that
\[
\tilde{E}(u,0,R_j^\frac{1}{2})\geq C_2 ,\ \ j\geq 1,
\]
for some $C_2>0$ ($C_2<C_1$ without loss of generality). In passing, we note that our motivation
for the power $1/2$ comes from the proof of the $\eta$-compactness lemma in \cite{riviere}.
By Lemma \ref{lemSmets}, we infer that there exist $r_j\in (R_j^\frac{1}{2},R_j)$ such that
\begin{equation}\label{eqsmets2}\begin{array}{rcl}
    r_j\int_{\partial B\left(0,r_j \right)}^{}
\left|\frac{\partial u}{\partial \nu} \right|^2 dS(y)+2\int_{B(0,r_j)}^{}W(u)dy& \leq &  \frac{2}{\ln R_j}\ln\left(\frac{C_1  (\ln R_j)^k}{C_2} \right)r_j^{n-2}\tilde{E}(u,0,r_j) \\
     &   &  \\
   \textrm{using}\ (\ref{eqmonotoniWeak}): &\leq& \frac{2}{\ln R_j}\ln\left(\frac{C_1  (\ln R_j)^k}{C_2} \right)C_1r_j^{n-2}  (\ln R_j)^k\\
 &   &  \\
 & \leq  & C_3 (\ln R_j)^{k-\frac{1}{2}} r_j^{n-2} \\
 &   &  \\
 & \leq   & C_4  (\ln r_j)^{k-\frac{1}{2}} r_j^{n-2},
  \end{array}
\end{equation}
for some $C_3,C_4>0$ and all $j\geq 1$.

Let $F\in C^\infty(\mathbb{R}^m;\mathbb{R}^m)$ be such that
\begin{equation}\label{eqF}
F(v)=v-a_i\ \ \textrm{if}\ \ |v-a_i|\leq \delta,\ \ i=1,\cdots,N,
\end{equation}
for some small $\delta>0$ (so that the $\delta$-neighborhoods of the $a_i$'s are disjoint).
We note that such a function can be constructed by first defining it in a $\delta$-neighborhood  of each $a_i$
and then extending it componentwise.
For future reference, let us note at this point that, by virtue of (\ref{eqpot1}) and (\ref{eqapriori}),  there exists a constant $C_5>0$ such that
\begin{equation}\label{eqcrucial}
\left|F(u)\cdot W_u(u)\right|\leq C_5 W(u), \ \ x\in \mathbb{R}^n.
\end{equation}
Taking the inner product of (\ref{eqEq}) with $F(u)$ and integrating by parts the resulting identity over $B(0,r_j)$ yields that
\begin{equation}\label{eqcombo1}\begin{array}{rcl}
    \int_{B(0,r_j)}^{}   F_u(u)  \nabla u  \cdot \nabla u dy & = & \int_{\partial B(0,r_j)}^{}\frac{\partial u}{\partial \nu} \cdot F(u) dS(y)-\int_{B(0,r_j)}^{}F(u)\cdot  W_u(u)dy \\
      &   &   \\
 \textrm{using}\ (\ref{eqapriori}),\ (\ref{eqcrucial}):     & \leq &  C_6r_j^{\frac{n-1}{2}}\left(\int_{\partial B(0,r_j)}^{} \left| \frac{\partial u}{\partial \nu}\right|^2dS(y)\right)^\frac{1}{2}+C_5\int_{B(0,r_j)}^{}W(u)dy
 \\
      &   & \\
\textrm{using}\ (\ref{eqsmets2}):& \leq &C_{7}(\ln r_j)^{k-\frac{1}{2}}r_j^{n-2},
  \end{array}
\end{equation}
for some $C_6, C_{7}>0$ and all $j\geq 1$.

Let
\[
\mathcal{A}_j=\left\{x\in B(0,r_j)\ :\ \left|u(x)-a_i\right|> \delta,\ \ i=1,\cdots,N \right\}.
\]
It follows from the first part of (\ref{eqpot1}), (\ref{eqapriori}) and (\ref{eqsmets2}) that
the $n$-dimensional Lebesgue measure of $\mathcal{A}_j$ satisfies
\[
\mathcal{H}^n(\mathcal{A}_j)\leq C_{8}(\ln r_j)^{k-\frac{1}{2}} r_j^{n-2},
\]
for some $C_{8}>0$ and all $j\geq 1$.
Therefore, using once more (\ref{eqapriori}), and (\ref{eqF}), we obtain that
\[
\begin{array}{rcl}
  \int_{B(0,r_j)}^{}   F_u(u)  \nabla u  \cdot \nabla u dy & \geq &
\int_{B(0,r_j)\setminus \mathcal{A}_j}^{}     |\nabla u|^2 dy-C_{9}(\ln r_j)^{k-\frac{1}{2}} r_j^{n-2} \\
    &   &   \\
    & \geq  & \int_{B(0,r_j) }^{}     |\nabla u|^2 dy-C_{10}(\ln r_j)^{k-\frac{1}{2}} r_j^{n-2},
\end{array}
\]
for some $C_{9},\ C_{10}>0$ and all $j\geq 1$.
Hence, it follows from (\ref{eqcombo1}) that
\[
\int_{B(0,r_j)}^{}     |\nabla u|^2 dy\leq C_{11}(\ln r_j)^{k-\frac{1}{2}} r_j^{n-2},
\]
for some $C_{11}>0$ and all $j\geq 1$.

By combining the above relation with (\ref{eqsmets2}), we arrive at
\[
\tilde{E}(u,0,r_j)\leq C_{12}(\ln r_j)^{k-\frac{1}{2}},
\]
for some $C_{12}>0$ and all $j\geq 1$.

We have therefore reduced the exponent in (\ref{eqoriginal}) by $1/2$ (for a different sequence $r_j\to \infty$). Iterating this scheme a finite number of times, we arrive at
\[
\tilde{E}(u,0,s_j)\to 0\ \ \textrm{as}\ \ j\to \infty,
\]
for some sequence $s_j\to \infty$.

On the other hand, the weak monotonicity formula (\ref{eqmonotoniWeak}) (recall also (\ref{eqmonotSmets})) implies that
$u\equiv a_i$ for some $i\in \{1,\cdots,N \}$, which  contradicts  the assumption that $u$ is nonconstant.
\end{proof}

\textbf{Acknowledgements.} Supported by the ``Aristeia'' program of the Greek
Secretariat for Research and Technology.

\end{document}